\documentclass[reqno,11pt]{amsart}

\usepackage{amsmath,amsfonts,amssymb,amsthm,epsfig}
\usepackage{txfonts}% this is added for the symbol of average integral.

\usepackage{verbatim}%%%to cancell paragraph

\usepackage{color}

\usepackage[final,allcolors=blue,colorlinks=true]{hyperref}
\usepackage[leqno]{amsmath}

\def\R{\mathbb R}

\def\cal{\mathcal}

\def\la{\lambda}

\newcommand{\medint}{-\kern -,375cm\int}         % This is the symbol of average integral
\newcommand{\medintinrigo}{-\kern -,315cm\int}

\def\Om{\Omega}

\def\pa{\partial}

\def\div{{\rm div}\,}

\newcommand\esssup{{\rm \,esssup\,}}

\numberwithin{equation}{section} \textwidth15cm \textheight22cm
\newtheorem{theorem}{Theorem}[section]

\newtheorem{lemma}[theorem]{Lemma}

\newtheorem{proposition}[theorem]{Proposition}

\theoremstyle{definition}
\begingroup

\endgroup

\begin{document}

\title[Asymptotic behaviors of solutions to quasilinear equation]{Asymptotic Behaviors of Solutions to  quasilinear elliptic
Equations with critical Sobolev growth and  Hardy potential}
\author{ Chang-Lin Xiang }
%\address[Daomin Cao]{Institute of Applied Mathematics, Academy of Mathematics and Systems Science,Chinese Academy of Sciences, Beijing 100190, China}
%\email[Daomin Cao]{dmcao@amt.ac.cn}
\address[]{Department of Mathematics and Statistics, P.O. Box 35 (MaD) FI-40014 University of Jyv\"askyl\"a, Finland}
 \email[]{changlin.c.xiang@jyu.fi}
%\address[Shusen Yan]{Department of Mathematics, The University of New England, Armidale, NSW, 2351, Australia} \email[Shusen Yan]{syan@turing.une.edu.au}

\begin{abstract}
Optimal estimates on the asymptotic behaviors  of weak solutions both at the  origin and  at  the  infinity are obtained
to the following quasilinear elliptic equations
\[-\Delta_{p}u-\frac{\mu}{|x|^{p}}|u|^{p-2}u=Q(x)|u|^{\frac{Np}{N-p}-2}u,\quad\,x\in \mathbb{R}^{N},\]
where $1<p<N,
0\leq\mu<\left({(N-p)}/{p}\right)^{p}$ and $Q\in
L^{\infty}(\mathbb{R}^{N})$.
\end{abstract}

\maketitle

{\small


\smallskip

}                 %%%%%%%%%%%     Contents
\bigskip
\section{Introduction and main results}

Let $1<p<N$, $ p^{*}=Np/{(N-p)}$ and $\
0\leq\mu<\bar{\mu}=\left((N-p)/{p}\right)^{p}$. In this
paper, we study the following quasilinear elliptic equations
\begin{equation}
-L_{p}u\equiv-\Delta_{p}u-\frac{\mu}{|x|^{p}}|u|^{p-2}u=Q(x)|u|^{p^{*}-2}u,
\quad\, x\in \mathbb{R}^{N}, \label{eq:object}
\end{equation}
where
\begin{eqnarray*}
\Delta_{p}u=\sum_{i=1}^{N}\partial_{x_{i}}(|\nabla
u|^{p-2}\partial_{x_{i}}u),&&\nabla
u=(\partial_{x_{1}}u,\cdots,\partial_{x_{N}}u)
\end{eqnarray*}
is the $p$-Laplacian operator and $Q\in
L^{\infty}(\mathbb{R}^{N})$. It is well known that equation
(\ref{eq:object}) is the Euler-Lagrange equation of the energy
functional $E:\mathcal{D}^{1,p}(\mathbb{R}^{N})\to\R$ defined by
\[
E(u)=\frac{1}{p}\int_{\mathbb{R}^{N}}\left(|\nabla
u|^{p}-\mu\frac{|u|^{p}}{|x|^{p}}\right)-\frac{1}{p^{*}}\int_{\mathbb{R}^{N}}Q|u|^{p^{*}},\quad
u\in \mathcal{D}^{1,p}(\mathbb{R}^{N}),
\]
where  $\mathcal{D}^{1,p}(\mathbb{R}^{N})$ is the function space
defined as
\[
\mathcal{D}^{1,p}(\mathbb{R}^{N})=\left\{v\in
L^{p^{*}}(\mathbb{R}^{N}):v\:\text{is weakly differentiable
and}\:\nabla v\in L^{p}(\mathbb{R}^{N})\right\}
\]
equipped with the seminorm
$\ensuremath{||v||_{\mathcal{D}^{1,p}(\mathbb{R}^{N})}=||\nabla
v||_{L^{p}(\mathbb{R}^{N})}}$. All of the integrals in energy
functional $E$ are well defined, due to the Sobolev inequality
\begin{eqnarray*}
C\left(\int_{\mathbb{R}^{N}}|\varphi|^{p^{*}}\right)^{\frac{p}{p^{*}}}\leq
\int_{\mathbb{R}^{N}}|\nabla\varphi|^{p},&&\forall\:\varphi\in\mathcal{D}^{1,p}(\mathbb{R}^{N}),
\end{eqnarray*} where $C=C(N,p)$
 is a positive constant, and  due to the Hardy inequality (see  \cite[Lemma 1.1]{B})
\begin{eqnarray*}
\left(\frac{N-p}{p}\right)^p\int_{\mathbb{R}^{N}}\frac{|\varphi|^{p}}{|x|^{p}}\leq\int_{\mathbb{R}^{N}}
|\nabla\varphi|^{p},&&\forall\:\varphi\in\mathcal{D}^{1,p}(\mathbb{R}^{N}).
\end{eqnarray*}

 A function $u\in
\mathcal{D}^{1,p}(\mathbb{R}^{N})$ is a weak subsolution of
equation (\ref{eq:object}) if for any nonnegative function $\varphi\in
C_{0}^{\infty}(\mathbb{R}^{N})$ we have
\[
\int_{\mathbb{R}^{N}}\left(|\nabla u|^{p-2}\nabla
u\cdot\nabla\varphi-\frac{\mu}{|x|^{p}}|u|^{p-2}u\varphi\right)\leq
\int_{\mathbb{R}^{N}}Q(x)|u|^{p^{*}-2}u\varphi.
\]
A function  $u$ is  a supersolution of  equation
(\ref{eq:object}) if $-u$ is a subsolution,  and $u$  a weak
solution of equation (\ref{eq:object}) if $u$ is both a weak
subsolution and a weak supersolution.

Equation (\ref{eq:object}) and its variants have been extensively
studied in the literature. For the existence of solutions to
equation (\ref{eq:object}), we refer to e.g.
\cite{Peral2004,B,Cao,Cao2,Cao4, Cao3, C,E,F,Veron1989,H,Smets}.
For the  uniqueness  of solutions to equation (\ref{eq:object}),
we refer to e.g. \cite{Caffarelli1989,C,ChouKS1993,GNN,Terracini}.
In the present paper, we study the asymptotic behaviors of weak
solutions to equation (\ref{eq:object}).

In the case $\mu=0$, a prototype of equation (\ref{eq:object})
(when $Q\equiv 1$) is
\begin{eqnarray}
-\Delta_{p}u=|u|^{p^{*}-2}u, &  &
\hbox{in}\,\,\mathbb{R}^{N}.\label{eq: mu is 0}
\end{eqnarray}
The boundedness of weak solutions to equation (\ref{eq: mu is 0})
in the neighborhood of the origin is well known. As to the
asymptotic behavior of solutions at the infinity, when $p=2$, it
was proved by Gidas, Ni and Nirenberg \cite{GNN} that positive
$C^2$ solutions  (not necessarily in the space
$\cal{D}^{1,2}(\R^N)$) of equation (\ref{eq: mu is 0}) satisfying
\begin{equation}\liminf_{|x|\to
\infty}\left(|x|^{N-2}u(x)\right)<\infty,\label{eq: GNN
condition}\end{equation} must be of the form $ u(x)=\lambda^{\frac
{N-2}2}u_0(\lambda(x-x_0))$ for some $\lambda
>0$ and $x_0\in\mathbb{R}^N$, where
\begin{equation}u_0(x)=(N(N-2))^{\frac {N-2}4}\left({1+|x|^2}
\right)^{-\frac {N-2}2}.\label{eq: exact form of solution
1}\end{equation} Hypothesis (\ref{eq: GNN condition}) was removed
by Caffarelli, Gidas and Spruck in \cite{Caffarelli1989}.
 Thus for positive $C^2$
solutions of equation (\ref{eq: mu is 0}),  we have
\begin{eqnarray*}|u(x)|\le C|x|^{2-N} && \hbox{for}\: |x|> 1\end{eqnarray*} for some positive constant $C$.
 This estimate has been proved to be true for  all weak
solutions in $\cal{D}^{1,2}(\R^N)$ of (\ref{eq: mu is 0}), see Cao
and Yan \cite{Cao3}.
 For $p\ne2$, Cao, Peng and Yan
\cite{Cao4} proved that for any weak solutions $u\in
\cal{D}^{1,p}(\R^N)$ of equation \eqref{eq: mu is 0}  we have
\begin{eqnarray}
|u(x)|\le C|x|^{-\frac{N-p}{p-1}+\theta} &
&\hbox{for}\: |x|> 1\label{eq: C-P-Y p not equal to 2}
\end{eqnarray}  for any $\theta>0$ and  some
positive constant $C$ (depending also on $\theta$).
 We remark that
 their result can be easily
extended,  by the same approach in \cite{Cao3}, to   equation
(\ref{eq:object}) ($\mu=0$) in the presence of a bounded function
$Q$.

We will focus on the case $\mu\ne0$. When $p=2$,
 the behavior of weak
solutions to equation (\ref{eq:object})  at origin is known. It was proved
that if $u$ is a weak solution of equation (\ref{eq:object}), then
\begin{eqnarray}
|u(x)|\le
C|x|^{-\left(\sqrt{\overline{\mu}}-\sqrt{\overline{\mu}-\mu}\right)} &&\hbox{for}\:
|x|<1\label{eq: estimate for p=2 at origin}
\end{eqnarray}
for some positive constant $C$, see \cite[Theorem 1.1]{Cao2}.  In
addition, if the function $Q$ is nonnegative  and the solution $u$
is also nonnegative,    Han \cite{H2} proved that
\begin{eqnarray}
u(x)\ge
C|x|^{-\left(\sqrt{\overline{\mu}}-\sqrt{\overline{\mu}-\mu}\right)}&&\hbox{for}\:
|x|<1\label{eq: inverse result for p=2}
\end{eqnarray} for some
nonnegative constant $C$. In fact,  Ferrero and Gazzola \cite{F}
proposed the problem of studying the asymptotic behavior  of
eigenfunctions of the operator $-L_{2}$ on bounded domain,
that is, to study the asymptotic behavior  of weak
solutions  to the linear equation as $x\to 0$
\begin{eqnarray}
-L_2 u\equiv -\Delta u-\frac{\mu}{|x|^{2}}u=\lambda u,&& u\in
H_{0}^{1}(\Omega),\label{eq: eigenvalue problem}
\end{eqnarray}
where $\lambda\in\mathbb{R}$ and $\Omega$ is a bounded domain
containing the origin. Cabr\'e and Martel \cite{Cab} obtained
\eqref{eq: estimate for p=2 at origin} for  the first
eigenfunctions in the case when $\Omega$ is a unit ball. Cao and
Han \cite{Cao}  proved \eqref{eq: estimate for p=2 at origin}  for
all solutions of equation (\ref{eq: eigenvalue problem}).  The
approach of \cite{Cao} is as follows: if $u$ is a weak solution to
equation \eqref{eq: eigenvalue problem}, then the function
$v(x)\equiv|x|^{\sqrt{\overline{\mu}}-\sqrt{\overline{\mu}-\mu}}u(x)$
satisfies the equation
\begin{eqnarray*}
-\div\left(|x|^{-2\left(\sqrt{\overline{\mu}}-\sqrt{\overline{\mu}-\mu}\right)}\nabla
v\right)=\lambda |x|^{-2\left(\sqrt{\overline{\mu}}-\sqrt{\overline{\mu}-\mu}\right)}v&&x\in\Omega,\label{eq: CKN}
\end{eqnarray*}
which is a weighted elliptic equation in divergence form.  By
means of   Moser's iteration technique \cite{Moser1960}, $v$ is
proved to be bounded, which is equivalent to (\ref{eq: estimate
for p=2 at origin}). In \cite{H2}, the author applied the same
method to deal with more general type of  equation  than equation
(\ref{eq:object}) when $p=2$, and proved the estimates (\ref{eq:
estimate for p=2 at origin}) and \eqref{eq: inverse result for
p=2}.  Obviously, the approach of \cite{Cao} is not
 applicable to  general quasilinear equations when $p\ne 2$.

As to asymptotic behaviors of solutions of equation
(\ref{eq:object}) at infinity when $p=2$, to the best of the author's knowledge,  all known results are concerned with the particular
case $Q\equiv1$. That is,  consider the equation
\begin{eqnarray}
-L_2 u\equiv -\Delta u-\frac {\mu}{|x|^2}u=|u|^{2^{*}-2}u, &  &
\hbox{in}\,\,\mathbb{R}^{N}.\label{eq: object of p equal to 2}
\end{eqnarray} For any positive solution $u\in C^2(\R^N\backslash\{0\})$ of equation
(\ref{eq: object of p equal to 2}) satisfying \begin{equation}
v(x)\equiv |x|^{\sqrt{\overline{\mu}}-\sqrt{\overline{\mu}-\mu}}u(x)
\in L^{\infty}_{loc}(\R^N),
\label{eq: boundedness condition }\end{equation} i.e., the
function $v$
 is
locally bounded in $\R^N$,
then  a direct calculation verifies that  $v$ satisfies  the
conditions of Theorem B of Chou and Chu \cite{ChouKS1993}, and
thus $v$ is radially symmetric with respect to the origin by
\cite[Theorem B]{ChouKS1993}. Therefore $u$ is radially symmetric with
respect to the origin.
 Catrina and Wang \cite{C} and Terracini
\cite{Terracini} proved that all positive radial solutions of
(\ref{eq: object of p equal to 2}) are of the form
$u(x)=\lambda^{\frac{N-2}{2}}u_0(\lambda  x)$ for some
$\lambda>0$,
where\begin{equation}u_0(x)=\left({4N(\bar{\mu}-\mu)}/{(N-2)}\right)^{\frac
{N-2}4}\left(
|x|^{\frac{\sqrt{\overline{\mu}}-\sqrt{\overline{\mu}-\mu}
}{\sqrt{\bar{\mu}}}}+|x|^{\frac{\sqrt{\overline{\mu}}+\sqrt{\overline{\mu}-\mu}
}{\sqrt{\bar{\mu}}}}\right)^{-\frac{N-2}2}.\label{eq: exact form
of solution 2}\end{equation}
 Thus for any positive solution $u\in C^2(\R^N\backslash\{0\})$ to
equation (\ref{eq: object of p equal to 2}) satisfying \eqref{eq:
boundedness condition }, there is a $\la>0$ such that $u(x)=\lambda^{\frac{N-2}{2}}u_0(\lambda  x)$. Consequently   we have
\begin{eqnarray*}
|u(x)|\le
C|x|^{-\left(\sqrt{\overline{\mu}}+\sqrt{\overline{\mu}-\mu}\right)}&&\hbox{for}\:
|x|> 1.\label{eq: estimate for p equal to 2 at infinity}
\end{eqnarray*} for some positive constant $C$.
 Cao and Yan \cite[Lemma B.2]{Cao3} proved this estimate for all weak
 solutions in $\cal{D}^{1,2}(\R^N)$ of equation \eqref{eq: object of p equal to 2}. Their method depends on  the Kelvin transformation
$v(x)=|x|^{2-N}u\left(|x|^{-2}{x}\right)$. It seems that this
approach does not work for general quasilinear equations
\eqref{eq:object} when $p\ne 2$.

Much less is known in  the case $\mu\ne0$   and  $p\ne2$.
Boumediene, Veronica and Peral \cite[Theorem 3.13]{B} classified
all weak
 positive \emph{radial} solutions in $\cal{D}^{1,p}(\R^N)$ of equation
(\ref{eq:object}) when $Q\equiv 1$. They are of the form
$u(x)=\lambda^{\frac {N-p}p}u_0(\lambda x)$ for some $\lambda>0$,
where   $u_0$ is a particular weak positive radial solution in
$\cal{D}^{1,p}(\R^N)$ (see \cite[Theorem 3.13]{B}) satisfying
\begin{eqnarray}
\lim_{|x|\rightarrow0}u_0(x)|x|^{\gamma_{1}}=C_{1},&&\lim_{|x|\rightarrow
\infty}u_0(x)|x|^{\gamma_{2}}=C_{2}\label{est:optimal result}
\end{eqnarray} for some positive constants $C_{1},C_{2}$.
 In (\ref{est:optimal result}) and throughout the paper,   $\gamma_{1},\gamma_{2}\in[0,\infty), \gamma_1<\gamma_2$,  are defined as the two roots of the
equation
\begin{equation}\label{eq: gamma 1 and gamma 2}
\gamma^{p-2}[(p-1)\gamma^{2}-(N-p)\gamma]+\mu=0.
\end{equation} While the exact form  (\ref{eq: exact form of solution 2}) of the positive radial solutions to equation (\ref{eq:object}) is known
when $p= 2$ and $Q\equiv 1$,    the exact form of the positive radial  solution  $u_0$ to  equation (\ref{eq:object})
 when $p\ne 2$ and $Q\equiv 1$ seems to be unknown.

For later use, we note that
\[
0\leq\gamma_{1}<\frac{N-p}{p}<\gamma_{2}\leq\frac{N-p}{p-1}.
\]
In the case $p=2$,
${\gamma_{1}}=\sqrt{\overline{\mu}}-\sqrt{\overline{\mu}-\mu}$ and
${\gamma_{2}}=\sqrt{\overline{\mu}}+\sqrt{\overline{\mu}-\mu}$,
and in the case $\mu=0$, $\gamma_{1}=0$ and $\gamma_{2}=(N-p)/(p-1)$.

In this paper, we give a complete description on the asymptotic
behaviors of  weak solutions to equation (\ref{eq:object}) at the
origin and at the infinity.
\begin{theorem}
\label{thm:main result} Let $Q\in
L^{\infty}(\mathbb{R}^{N})$ and $u\in\mathcal{D}^{1,p}(\mathbb{R}^{N})$  be a weak solution of equation (\ref{eq:object}).  Then there exists a positive  constant $C$ depending on $N,p,\mu,||Q||_{\infty}$ and $u$, such that
\begin{equation}
|u(x)|\leq {C}{|x|^{-\gamma_{1}}}\quad for\;|x|<
R_{0},\label{eq:origin growth}
\end{equation}
 and
\begin{equation}
|u(x)|\leq{C}{|x|^{-\gamma_{2}}}\quad for\;|x|>
R_{1},\label{eq:infinity decay}
\end{equation}
where $R_{0}>0$ and $R_{1}>0$ depend on
$N,p,\mu,||Q||_{\infty}$ and $u$.\end{theorem}
In the above theorem, the positive constants $C,R_0,R_1$ depend on the
solution $u$. Indeed, this is the case, since equation \eqref{eq:object} when $Q\equiv 1$
 is invariant under the scaling $v(x)=\lambda^{\frac {N-p}p}u(\lambda x)$, $\lambda >0$.
 In above theorem and in the following, if we say a constant depends on the solution $u$,
  it means that the constant depends on $||u||_{L^{p^{*}}(\mathbb{R}^{N})}$,
  the $L^{p^*}-$norm of $u$, and also on the modulus of continuity of the function
   $h(\rho)=||u||_{L^{p^{*}}(B_{\rho}(0))}+||u||_{L^{p^{*}}(\mathbb{R}^{N}\backslash
B_{1/\rho}(0))}$ at zero. Precisely, we will choose a
   constant $\epsilon_0>0$ depending on $N,p,\mu,||Q||_{\infty}$. Since  $h(\rho)\to 0$
   as $\rho\to 0$,  there exists $\rho_0 >0$ such that \[
||u||_{L^{p^{*}}(B_{\rho_0}(0))}+||u||_{L^{p^{*}}(\mathbb{R}^{N}\backslash
B_{1/\rho_0}(0))}<\epsilon_0.
\] The constants $C,R_0,R_1$ in Theorem \ref{thm:main result} depend  on $\rho_0$.

We remark that our result is new even in the case $p\ne 2,\mu=0$. We improve
the estimate \eqref{eq: C-P-Y p not equal to 2} by Cao, Peng and Yan \cite{Cao4}.

The following theorem shows that the exponents $\gamma_1$ and $\gamma_2$ in the
 estimates \eqref{eq:origin growth} and (\ref{eq:infinity decay}) respectively in Theorem \ref{thm:main result} are sharp.
\begin{theorem}\label{thm: inverse of main resuls}
Let  $Q\in
L^{\infty}(\mathbb{R}^{N})$ be a nonnegative function and $u\in\mathcal{D}^{1,p}(\mathbb{R}^{N})$   a nonnegative  weak solution of equation (\ref{eq:object}). Then
\begin{equation}
u(x)\ge
m|x|^{-\gamma_{1}}\quad\text{for}\;|x|<1,\label{eq:origin
growth 1}
\end{equation}
and
\begin{equation}
u(x)\ge
m|x|^{-\gamma_{2}}\quad\text{for}\;|x|>1,\label{eq:infinity
decay 2}
\end{equation}
where $m=\inf_{\partial B_{1}(0)}u$.
\end{theorem}

In fact, Theorem \ref{thm:main result} and Theorem \ref{thm: inverse of main resuls} hold for solutions of more general equations. Consider the equation
\begin{eqnarray} &-L_{p}u= f|u|^{p-2}u  &
\label{eq: general equation}
\end{eqnarray}
in an arbitrary domain $\Omega\subset\mathbb{R}^N$, where $f$ is a
function in the space $L^{\frac Np}(\Omega)$. Equation (\ref{eq: general equation}) is the
 Euler-Lagrange equation of the energy functional $E:\mathcal{D}^{1,p}(\Omega)\to \R$ defined by
\[
E(u)=\frac{1}{p}\int_{\Omega}\left(|\nabla
u|^{p}-\mu\frac{|u|^{p}}{|x|^{p}}\right)-\frac{1}{p}\int_{\Omega}f|u|^{p},\quad
u\in \mathcal{D}^{1,p}(\Omega),
\]
where
$\mathcal{D}^{1,p}(\Omega)$ is   the function
space defined as
$$\mathcal{D}^{1,p}(\Omega)=\left\{v\in L^{p^*}(\Omega):v\:\text{is weakly
differentiable and}\: \nabla v\in L^p(\Omega) \right\},$$ equipped
with the seminorm $||v||_{\mathcal{D}^{1,p}(\Omega) }=||\nabla
v||_{L^{p}(\Omega)}$.

A function $u\in \mathcal{D}^{1,p}(\Omega)$
is  a weak subsolution
 of equation (\ref{eq: general equation}) in
$\Omega$ if
\begin{equation*}
\int_{\Omega}\left(|\nabla u|^{p-2}\nabla
u\cdot\nabla\varphi-\frac{\mu}{|x|^{p}}|u|^{p-2}u\varphi\right)\leq
\int_{\Omega}f|u|^{p-2}u\varphi\label{eq: definition of weak
solutions}
\end{equation*}
 for all nonnegative function $\varphi\in C^{\infty}_0(\Omega)$. A function $u$ is  a weak
supersolution
 of equation (\ref{eq: general equation}) in
$\Omega$ if $-u$ is  a weak subsolution. A function $u$ is  a weak
solution
 of equation (\ref{eq: general equation}) in
$\Omega$ if $u$ is  both a weak subsolution
 and a weak supersolution.

The following theorem gives  the asymptotic behavior of solutions
to equation (\ref{eq: general equation}) at the origin.

\begin{theorem}
\label{thm:given growth comparison at origin 1} Let $\Omega$ be a
bounded  domain containing the origin and  $f$  a  function in $ L^{\frac
Np}(\Omega)$ satisfying
$f(x)\leq A|x|^{-\alpha}$ in $\Omega$ for some
given constants $A,\alpha$, $A>0$ and  $p>\alpha$. If $u\in
\mathcal{D}^{1,p}(\Omega)$ is a weak subsolution to equation
(\ref{eq: general equation}) in $\Omega$, then
there exists a positive
constant $C$  depending  on $N,p,\mu,A$ and $\alpha$ such that
\begin{eqnarray*} u(x)\leq C M |x|^{-\gamma_{1}}&&
\hbox{for}\:x\in B_{ R_0}(0),
\end{eqnarray*}
where $M=\sup_{\partial B_{R_0}(0)}u^+$ and $R_0>0$ is a constant depending  on $N,p,\mu,A,\alpha$.
\end{theorem}
We also have the following theorem which shows that the exponent $\gamma_1$ in Theorem
\ref{thm:given growth comparison at origin 1}
is optimal.
\begin{theorem}
\label{thm:given growth comparison at origin 2} Let $\Omega$ be a
bounded  domain containing the origin and  $f\in L^{\frac
Np}(\Omega)$  a nonnegative  function. If   $ u\in
\mathcal{D}^{1,p}(\Omega)$ is  a nonnegative weak supersolution of
equation (\ref{eq: general equation}) in $\Omega$, then
\begin{eqnarray*}u(x)\ge  Cm
{|x|}^{-\gamma_1}& & \hbox{for}\:x\in\Omega,\end{eqnarray*} where
$m=\inf_{\Omega}u$ and $C=\inf_{\partial\Omega}|x|^{\gamma_1}$.
\end{theorem}
We also have the following corresponding results on the asymptotic behavior of weak
solutions of equation \eqref{eq: general equation} at infinity.

\begin{theorem}
\label{thm:given growth comparison at infinity 1}  Let $\Omega$ be
an exterior domain in $\mathbb{R}^N$ such that
$\Omega^c=\mathbb{R}^N\backslash \Omega$ is bounded and function
$f\in L^{\frac Np}(\Omega)$ satisfy
 $f(x)\leq A|x|^{-\alpha}$ in $\Omega$ for some
given constants $A,\alpha$,  $A>0$, $p<\alpha$. If $u\in
\mathcal{D}^{1,p}(\Omega)$ is a weak subsolution of equation
(\ref{eq: general equation}) in $\Omega$, then there exists
 a positive
constant $C$ depending only on $N,p,\mu,A$ and $\alpha$
such that
\begin{eqnarray*} u(x)\leq C M |x|^{-\gamma_{2}}&&
\hbox{for}\:x\in \R^N\backslash B_{ R_1}(0),
\end{eqnarray*}
where $M=\sup_{\partial B_{R_1}(0)}u^+$ and
 $R_1>1$ is a constant depending on
$N,p,\mu,A,\alpha$. \end{theorem}

\begin{theorem}
\label{thm:given growth comparison at infinity 2}  Let $\Omega$ be
an exterior domain in $\mathbb{R}^N$ such that
$\Omega^c=\mathbb{R}^N\backslash \Omega$ is bounded and
$f\in L^{\frac Np}(\Omega)$ a  nonnegative function. If  $ u\in
\mathcal{D}^{1,p}(\Omega)$ is  a nonnegative weak supersolution of
equation (\ref{eq: general equation}) in $\Omega$, then
\begin{eqnarray*}u(x)\ge  Cm
{|x|}^{-\gamma_2}& & \hbox{for}\:x\in\Omega,\end{eqnarray*} where
$m=\inf_{ \partial \Omega}u$ and
$C=\inf_{\partial\Omega}|x|^{\gamma_2}$.
\end{theorem}

Before we close this section, we outline the proof of
 Theorem \ref{thm:given growth comparison at origin 2}. Other theorems are proved similarly. Suppose
 that $u\in
\mathcal{D}^{1,p}(\Omega)$ is a nonnegative supersolution to
equation (\ref{eq: general equation}) in  a bounded domain
$\Omega$ containing the origin and $f$ is a nonnegative function
in $L^{\frac Np}(\Omega)$. Note that the function
$v(x)=m|x|^{-\gamma_1}$, $m=\inf_{\Om}u$,  is a solution of
equation
\[-L_p v=0\qquad \hbox{in}\: \Omega.\]
Therefore, $v$ is also a subsolution of equation (\ref{eq: general
equation}).
 To prove Theorem \ref{thm:given growth comparison at origin 2}, we establish a comparison
  principle between subsolutions and supersolutions of equation (\ref{eq: general equation}) on
  $\Omega$. The comparison principle is known in the case when $\mu=0$ and $f\equiv 0$, see e.g. \cite{Lindqvist2}. The comparison principle
 in the general case  is established in Theorem \ref{thm: Comparison} in Section 3. Then Theorem
    \ref{thm:given growth comparison at origin 2} follows by verifying that the
     supersolution $u$ and subsolution $v$ satisfy all the conditions required
     in Theorem \ref{thm: Comparison}. Our idea to prove the
     comparison principle for equation (\ref{eq: general equation}) is inspired by the paper \cite{Lin} of Lindqvist,
     where he proved the simplicity of the first eigenvalue of
     the (minus) $p$-Laplacian operator. The essential point of
     his proof is to use the test functions of the type
     $\varphi=u^{1-p}(u^p-v^p)$ and $\psi=v^{1-p}(v^p-u^p)$.

The paper is organized as follows.  In Section 2, we establish some preliminary  asymptotic
behaviors estimates for solutions of equation (\ref{eq:object}),
and in Section 3 we establish comparison principles for
subsolutions and supersolutions of equation \eqref{eq: general equation} both on bounded domains and
exterior domains. Section  4 is devoted to   the proof of   theorems listed above.

 Throughout the paper, we denote domains by
$\Omega$,  the complement of $\Omega$ in $\R^N$ by $\Omega^c$.  We
also denote by $B_R$ or by $B_R(0)$ the ball centered at origin
with radius $R$. For any $1\le q\le\infty$, $L^{q}(\Omega)$ is the
Banach space of Lebesgue measurable functions $u$ such that the
norm
\[
||u||_{q,\Omega}=\begin{cases}
\left(\int_{\Omega}|u|^{q}\right)^{\frac{1}{q}}, & \hbox{if }\,1\le q<\infty,\\
\esssup_{\Omega}|f|, & \text{if }q=\infty
\end{cases}
\] is finite. 

\section{Preliminary estimates}

In this section we prove the following preliminary asymptotic
behavior estimates for solutions of equation (\ref{eq:object}). These estimates are used to prove Theorem \ref{thm:main result}.

\begin{proposition}
\label{thm:poor estimate}
 Let $Q\in
L^{\infty}(\mathbb{R}^{N})$ and $u\in\mathcal{D}^{1,p}(\mathbb{R}^{N})$  be a weak solution of equation (\ref{eq:object}).  Then there exists a positive  constant $C$ depending on $N,p,\mu,||Q||_{\infty}$ and $u$, such that
\begin{eqnarray*}
|u(x)|\leq C|x|^{-\left(\frac{N-p}{p}-\tau_{0}\right)} &  &
\text{for }|x|< r_0,
\end{eqnarray*}
 and that
\begin{eqnarray*}
|u(x)|\leq C|x|^{-\left(\frac{N-p}{p}+\tau_{1}\right)} &  &
\text{for}\: |x|> r_1,
\end{eqnarray*}
where  $\tau_{0},\tau_{1}, r_0, r_1 >0$ are  constants depending on $N,p,\mu, ||Q||_{\infty}$ and $u$.
\end{proposition}

These pointwise estimates will be proved by Moser's iteration argument \cite{Moser1960}.  We divide the proof into several lemmas.
\begin{lemma}
\label{lem: dependence on the solution } There exists a constant
$\epsilon_1=\epsilon_1(N,p,\mu,||Q||_{\infty})>0$ such that for any
solution $u\in\mathcal{D}^{1,p}(\mathbb{R}^{N})$ to equation
(\ref{eq:object}) and for any $\rho>0$ satisfying
\begin{equation}
||u||_{L^{p^{*}}(B_{\rho})}+||u||_{L^{p^{*}}(\mathbb{R}^{N}\backslash
B_{1/\rho})}\le\epsilon_1,\label{eq:smallness condition 1}
\end{equation} there exists a positive constant $C$ depending on $N,p,\mu,||Q||_{\infty},\rho$
 such that
\begin{eqnarray*}
||u||_{L^{p^{*}}(B_{R})}\leq CR^{\tau_{0}} &  & \text{for}\:
R\le\rho,
\end{eqnarray*}
and that
\begin{eqnarray*}
||u||_{L^{p^{*}}(B_{R}^{c})}\leq CR^{-\tau_{1}} &  & \text{for}\:
R\geq1/\rho,
\end{eqnarray*}
where $\tau_{0},\tau_{1}>0$  are two constants depending on $N,p,\mu$. \end{lemma}
Let $V=Q(x)|u|^{p^{*}-p}$. Equation (\ref{eq:object}) reads as \[
-L_pu=V|u|^{p-2}u\quad\hbox{in}\:\R^N.\]
\begin{proof}[proof of Lemma \ref{lem: dependence on the solution }]
We only prove the first inequality in Lemma \ref{lem: dependence
on the solution }. The second one can be proved similarly. Let
$R>0$ and $\eta\in C_{0}^{\infty}(\mathbb{R}^{N})$ be a cut-off
function such that $0\leq\eta\leq1$ in $\mathbb{R}^{N}$,
$\eta\equiv0$ on $B_{R}^{c}$, $\eta\equiv1$ on $B_{R/2}$ and
$|\nabla\eta|\leq4/R$. Substituting test function
$\varphi=\eta^{p}u$ into equation (\ref{eq:object}),  we obtain
\[
\langle-L_{p}u,\varphi\rangle\equiv\int_{\mathbb{R}^{N}}\left(|\nabla
u|^{p-2}\nabla
u\cdot\nabla\varphi-\frac{\mu}{|x|^{p}}|u|^{p-2}u\varphi\right)=\int_{\mathbb{R}^{N}}V|u|^{p-2}u\varphi.
\]
Note that for any $\delta\in(0,1)$ there is a constant
$C_{\delta}>0$ such that
\begin{eqnarray*}
\int_{\mathbb{R}^{N}}|\nabla u|^{p-2}\nabla u\cdot\nabla\varphi &
= & \int_{\mathbb{R}^{N}} \left(\eta^{p}|\nabla
u|^{p}+p\eta^{p-1}u|\nabla u|^{p-2}\nabla u\cdot\nabla\eta
 \right)\\& \geq & (1-\delta)\int_{\mathbb{R}^{N}}|\nabla(\eta u)|^{p}-C_{\delta}\int_{\mathbb{R}^{N}}|\nabla\eta|^{p}|u|^{p},
\end{eqnarray*}
and from the Hardy inequality we have
\[
\int_{\mathbb{R}^{N}}\frac{\mu}{|x|^{p}}|u|^{p-2}u\varphi\leq\frac{\mu}{\bar{\mu}}\int_{\mathbb{R}^{N}}|\nabla(\eta
u)|^{p}.
\]
Taking $\delta=\delta(N,p,\mu)>0$ small enough such that
$(1-\delta)>{\mu}/{\bar{\mu}}$, we have
\begin{eqnarray*}
\langle-L_{p}u,\varphi\rangle & \geq & (1-\delta-\mu/\bar{\mu})\int_{\mathbb{R}^{N}}|\nabla(\eta u)|^{p}-C_{\delta}\int_{\mathbb{R}^{N}}|\nabla\eta|^{p}|u|^{p}\\
 & \geq & C_{1}\left(\int_{\mathbb{R}^{N}}|\eta u|^{p^{*}}\right)^{\frac{p}{p^{*}}}-C_{2}\left(\int_{B_{R}\backslash B_{R/2}}|u|^{p^{*}}\right)^{\frac{p}{p^{*}}},
\end{eqnarray*}
where $C_{i}=C_{i}(N,p,\mu)$, $i=1,2$. On the other hand,
\[
\left|\int_{\mathbb{R}^{N}}V|u|^{p-2}u\varphi\right|\le||V||_{\frac{N}{p},B_{R}}\left(\int_{\mathbb{R}^{N}}|\eta
u|^{p^{*}}\right)^{\frac{p}{p^{*}}}\le||Q||_{\infty}||u||_{p^{*},B_{R}}^{p^{*}-p}\left(\int_{\mathbb{R}^{N}}|\eta
u|^{p^{*}}\right)^{\frac{p}{p^{*}}}.
\]
Thus we obtain that
\[
\left(\int_{\mathbb{R}^{N}}|\eta
u|^{p^{*}}\right)^{\frac{p}{p^{*}}}\le
C\left(\int_{B_{R}\backslash
B_{R/2}}|u|^{p^{*}}\right)^{\frac{p}{p^{*}}}+C||Q||_{\infty}||u||_{p^{*},B_{R}}^{p^{*}-p}\left(\int_{\mathbb{R}^{N}}|\eta
u|^{p^{*}}\right)^{\frac{p}{p^{*}}}.
\]
Set
\[
\epsilon_1=\left(4C||Q||_{\infty}\right)^{-1/(p^{*}-p)},
\]
and choose $\rho>0$ small enough such that (\ref{eq:smallness condition
1}) holds. Then
$C||Q||_{\infty}||u||_{p^{*},B_{R}}^{p^{*}-p}\le1/2$ for all
$0<R\le\rho$. We obtain
\begin{eqnarray*}
\int_{B_{R/2}}|u|^{p^{*}}\le\int_{\mathbb{R}^{N}}|\eta
u|^{p^{*}}\leq C\int_{B_{R}\backslash B_{R/2}}|u|^{p^{*}} &  &
\forall\;0<R\le\rho,
\end{eqnarray*}
where $C$ depends only on $N,p,\mu$. Denote
$\Psi(R)=\int_{B_{R}}|u|^{p^{*}}$ for $R\le\rho$. We get that
\begin{eqnarray*}
\Psi(R/2)\leq\theta\Psi(R), &  & \forall\:0<R\le\rho,
\end{eqnarray*}
where $\theta=\frac{C}{C+1}\in(0,1)$ depends only on $N,p,\mu$.
Now applying Lemma 3.5 in \cite[chapter 4]{LinFanghua} to $\Psi$
on the interval $[0,\rho]$, we obtain
\[
\Psi(R)\le\frac{1}{\theta}\Psi(\rho)\left(\frac{R}{\rho}\right)^{\tau_{0}^{'}}\quad\forall\:0<R\le\rho,
\]
where $\tau_{0}^{'}=\log(1/\theta)/\log2$ depends only on
$\theta$. Now the first inequality of this lemma follows by
setting $\tau_{0}=\tau_{0}^{'}/p^{*}$ and
$C=(\theta^{-1}\rho^{-\tau_{0}^{'}}\Psi(\rho))^{1/p^{*}}$.
\end{proof}
Recall that we denote by $B_R$ the ball $B_R(0)$ centered at origin with radius $R$. Let $A_{R}=B_{8R}\backslash
B_{R/8}$ and $D_{R}=B_{4R}\backslash B_{R/4}$ for $R>0$. We need
the following uniform estimate with respect to $R$.
\begin{lemma}
\label{pro:the average integral estimate} Let
$t\in(p^{*},N/\gamma_{1})$. Then there exists a constant
$\epsilon_2=\epsilon_2(N,p,\mu,||Q||_{\infty},t)>0$ such that for any
solution $u\in\mathcal{D}^{1,p}(\mathbb{R}^{N})$ to equation
(\ref{eq:object}) and for any $\rho>0$ satisfying
\begin{equation}
||u||_{L^{p^{*}}(B_{\rho})}+||u||_{L^{p^{*}}(\mathbb{R}^{N}\backslash
B_{1/\rho})}\le\epsilon_2,\label{eq:smallness condition 2}
\end{equation} there holds
\begin{equation}
\left(\fint_{D_{R}}|u|^{t}\right)^{\frac{1}{t}}\leq
C\left(\fint_{A_{R}}|u|^{p^{*}}\right)^{\frac{1}{p^{*}}},\quad\forall\,
R<\rho/8\text{ or }R>8/\rho,\label{eq:local improvement 1}
\end{equation}
where $\fint_{D_{R}}|u|^{t}=\frac{1}{|D_{R}|}\int_{D_{R}}|u|^{t}$ and $C>0$ is a constant depending on
$N,p,\mu,t,||Q||_{\infty}$ and $\rho$ but independent of $R$.
\end{lemma}
\begin{proof}
The proof is essentially the same as that of Theorem 1.3 in \cite[Chapter 4]{LinFanghua}. Here we need to show that  the constant $C$ in (\ref{eq:local improvement 1}) is independent of $R$.
Set $v(x)=u(Rx)$ for $R>0$. Then $v$ satisfies
\[
-L_{p}v=V_{R}|v|^{p-2}v\quad\hbox{in}\: A_{1},
\]
where $V_{R}(x)=R^{p}V(Rx)$.   Define $\bar{v}=\max(v,0)$ and $v_{m}=\min(\bar{v},m)$ for $m\ge 1$.  Substituting,
for any $\eta\in C_{0}^{\infty}(A_{1})$, % $0\leq\eta\leq1$,$\eta=1$ on $B_{R}$ and $|\nabla\eta|\leq2/R$
and $\frac{N-p}{p\gamma_{1}}>s\ge1,$  the
test function $\varphi=\eta^{p}v_{m}^{p(s-1)}\bar{v}$ into the
equation of $v$, we  get
\[
\langle-L_{p}v,\varphi\rangle\equiv\int_{A_1}\left(|\nabla
u|^{p-2}\nabla
u\cdot\nabla\varphi-\frac{\mu}{|x|^{p}}|u|^{p-2}u\varphi\right)=\int_{A_{1}}V_{R}|v|^{p-2}v\varphi.
\]
It is easy to see that for any $\delta>0$ small there exists
$C_{\delta}>0$ such that
\[
\int_{A_{1}}|\nabla v|^{p-2}\nabla
v\cdot\nabla\varphi\geq(1-\delta)\frac{p(s-1)+1}{s^{p}}\int_{A_{1}}|\nabla(\eta
v_{m}^{s-1}\bar{v})|^{p}-C_{\delta}\int_{A_{1}}|\nabla\eta|^{p}v_{m}^{p(s-1)}\bar{v}^{p},
\]
and from the Hardy inequality,
\[
\int_{A_{1}}\frac{\mu}{|x|^{p}}|v|^{p-2}v\varphi=\int_{A_{1}}\frac{\mu}{|x|^{p}}(\eta
v_{m}^{s-1}\bar{v})^{p}\leq\frac{\mu}{\bar{\mu}}\int_{A_{1}}|\nabla(\eta
v_{m}^{s-1}\bar{v})|^{p}.
\]
Since $\frac{p(s-1)+1}{s^{p}}>\frac{\mu}{\bar{\mu}}$ for all
$s\in\left(\frac{N-p}{p\gamma_{2}},\frac{N-p}{p\gamma_{1}}\right)$,
we can choose $\delta$ small enough such that
\begin{eqnarray*}
\langle-L_{p}v,\varphi\rangle
 & \geq & \left((1-\delta)\frac{p(s-1)+1}{s^{p}}-\frac{\mu}{\bar{\mu}}\right)\int_{A_{1}}|\nabla(\eta v_{m}^{s-1}\bar{v})|^{p}-C_{\delta}\int_{A_{1}}|\nabla\eta|^{p}v_{m}^{p(s-1)}\bar{v}^{p}\\
 & \geq & C_{1}\left(\int_{A_{1}}|\eta v_{m}^{s-1}\bar{v}|^{p\chi}\right)^{1/\chi}-C_{2}\int_{A_{1}}|\nabla\eta|^{p}v_{m}^{p(s-1)}\bar{v}^{p}
\end{eqnarray*}
for some constants $C_{1},C_{2}>0$ depending on $N,p,\mu,s$, where
$\chi=p^{*}/p$. On the other hand, H\"older's inequality gives us
\[
\int_{A_{1}}V_{R}|v|^{p-2}v\varphi\leq||V_{R}||_{\frac{N}{p},A_{1}}\left(\int_{A_{1}}|\eta
v_{m}^{s-1}\bar{v}|^{p\chi}\right)^{1/\chi}\le||Q||_{\infty}||u||_{p^{*},A_{R}}^{p^{*}-p}\left(\int_{A_{1}}|\eta
v_{m}^{s-1}\bar{v}|^{p\chi}\right)^{1/\chi}.
\]
Therefore we have
\begin{equation}
\left(\int_{A_{1}}|\eta
v_{m}^{s-1}\bar{v}|^{p\chi}\right)^{1/\chi}\leq
C_3\int_{A_{1}}|\nabla\eta|^{p}v_{m}^{p(s-1)}\bar{v}^{p}+C_3||u||_{L^{p^{*}}(A_{R})}^{p^{*}-p}\left(\int_{A_{1}}|\eta
v_{m}^{s-1}\bar{v}|^{p\chi}\right)^{1/\chi}\label{eq: 2.3.1}
\end{equation}
for some constant $C_3=C_3(N,p,\mu,||Q||_{\infty},s)>0$.

Fix $t\in(p^{*},N/\gamma_{1})$ and $k\in\mathbb{N}$ so that
$p\chi^{k}\le t<p\chi^{k+1}$.  Then there exists a positive constant
$C_{3}=C_{3}(N,p,\mu,||Q||_{\infty},t)$ such that (\ref{eq:
2.3.1}) holds for all $1\le s\le\min\left\{
\frac{N-p}{p\gamma_{1}},\chi^{k}\right\} $.

Set
\[
\epsilon_2\equiv\left(4C_{3}\right)^{-1/(p^{*}-p)}
\]
 and choose $\rho\in (0,1) $ small enought such that 
(\ref{eq:smallness condition 2}) holds. Then
\begin{eqnarray*}
||u||_{L^{p^{*}}(A_{R})}\le\epsilon_2, &  &
\forall\:0<R\le\rho/8\text{ or }R\ge8/\rho.
\end{eqnarray*}
Therefore, for all $0<R\le\rho/8\text{ or }R\ge8/\rho$, we have
\begin{eqnarray*}
\left(\int_{A_{1}}|\eta
v_{m}^{s-1}\bar{v}|^{p\chi}\right)^{1/\chi}\leq
C\int_{A_{1}}|\nabla\eta|^{p}v_{m}^{p(s-1)}\bar{v}^{p}, &  &
\end{eqnarray*}
for all $1\le s\le\min\left\{
\frac{N-p}{p\gamma_{1}},\chi^{k}\right\} $, where $C>0$ depends
only on $N,p,\mu,t,||Q||_{\infty},\rho$. The same procedure can be
applied also to $(-v)^{+}=\max(-v,0)$, the positive part of $-v$.

Now by choosing appropriate functions $\eta$ and then applying
Moser's iteration method \cite{Moser1960}, we conclude
that for any $t\in(p^{*},N/\gamma_{1})$ fixed, after finitely many
times of iteration, we can achieve the estimate
\begin{equation}
\left(\fint_{D_{1}}|v|^{t}\right)^{\frac{1}{t}}\leq
C\left(\fint_{A_{1}}|v|^{p^{*}}\right)^{\frac{1}{p^{*}}}.\label{eq:
uniform estimate for v}
\end{equation}
This proves  (\ref{eq:local improvement 1}).
\end{proof}
Now we prove  Proposition \ref{thm:poor estimate}.

\begin{proof}[Proof of Proposition \ref{thm:poor estimate}] Let
$t=(p^{*}+N/\gamma_{1})/2\in(p^{*},N/\gamma_{1})$ as in Lemma
\ref{pro:the average integral estimate} and
$\epsilon_0=\min(\epsilon_1,\epsilon_2)$, where $\epsilon_1$ and $\epsilon_2$ are as in Lemma \ref{lem: dependence on the
solution } and Lemma \ref{pro:the average integral estimate} respectively.   Let $\rho>0$ such that (\ref{eq:smallness condition 1})
holds for $\epsilon_0$. Still consider the equation of $v(x)=u(Rx)$ given by
\[
-\Delta_{p}v+c(x)|v|^{p-2}v=0\quad\text{in}\: D_{1},
\]
where $c(x)=-\mu|x|^{-p}-V_{R}(x)$. Note that $|x|^{-p}$ is a
bounded function on $D_{1}$, and $V_{R}(x)=R^pV(Rx)\in L^{q}(D_{1})$ with
$q=\frac{t}{p^{*}-p}>\frac{N}{p}$ due to \eqref{eq: uniform
estimate for v}. Following the proof of Theorem 1.1 in
\cite[Chapter 4]{LinFanghua} one gets that
\begin{equation}
\sup_{B}|v|\le
C\left(\fint_{2B}|v|^{p}\right)^{\frac{1}{p}}\label{eq:estimate
for rescaled solution}
\end{equation}
for any ball $B=B(x,r)$ such that $2B=B(x,2r)\subset D_{1}$, where
$C=C(N,p,\mu,||V_{R}||_{L^{q}(2B)})$.

We claim that the quantity $||V_{R}||_{L^{q}(D_{1})}$ is uniformly
bounded with respect to $R$. Moreover, there exists a constant
$C>0$ depending on $N,p,\mu,q,||Q||_{\infty}$ and $\rho$ such that
\[
||V_{R}||_{L^{q}(D_{1})}\leq
C||u||_{p^{*},\mathbb{R}^{N}}^{p^{*}-p}\quad\forall\:0<R\leq\rho/8\text{
or }R\ge8/\rho.
\]
Indeed, by the definition of $V_{R}$ and (\ref{eq:local improvement
1}),
\begin{eqnarray*}
||V_{R}||_{q,D_{1}} & = & R^{p-\frac{N}{q}}||V||_{q,D_{R}}\\
 & \le & ||Q||_{\infty}R^{p-\frac{N}{q}}||u||_{t,D_{R}}^{p^{*}-p}\\
 & \leq & CR^{p-\frac{N}{q}-\left(\frac{N}{p^{*}}-\frac{N}{t}\right)(p^{*}-p)}||u||_{p^{*},A_{R}}^{p^{*}-p}\\
 & \le & C||u||_{p^{*},\mathbb{R}^{N}}^{p^{*}-p},
\end{eqnarray*}
since $p-\frac{N}{q}-(\frac{N}{p^{*}}-\frac{N}{t})(p^{*}-p)=0$.
Therefore the estimate (\ref{eq:estimate for rescaled solution})
is uniform with respect to $0<R\leq\rho/8$ or $R\ge8/\rho$.

Now a simple covering argument leads us to
\[
\sup_{B_{2}\backslash B_{1}}|v|\leq
C\left(\fint_{D_{1}}|v|^{p}\right)^{\frac{1}{p}}.
\]
Recall that $v(x)=u(Rx)$.  Equivalently we arrive at
\[
\sup_{B_{2R}\backslash B_{R}}|u(x)|\le
C\left(\fint_{D_{R}}|u|^{p}\right)^{\frac{1}{p}}\quad\forall\:0<R\leq\rho/8\text{
or }R\ge8/\rho.
\]
Hence by H\"older's inequality,
\begin{equation}
\sup_{B_{2R}\backslash B_{R}}|u|\le
C\left(\fint_{D_{R}}|u|^{p^{*}}\right)^{\frac{1}{p^{*}}}=C||u||_{L^{p*}(D_{R})}R^{\frac{p-N}{p}},\quad\forall\:0<R\leq\rho/8\text{
or }R\ge8/\rho,\label{eq:maximum estimate by integral}
\end{equation}
 where $C$ depends only on $N,p,\mu,q,||Q||_{\infty},\rho$ and $||u||_{p^{*},\mathbb{R}^{N}}$.

Since $A_{R}\subset B_{\rho}$ for $0<R\leq\rho/8$ and
$A_{R}\subset B_{1/\rho}^{c}$ for $R\ge8/\rho$, by Lemma
\ref{lem: dependence on the solution },  there exist
$\tau_{0},\tau_{1}>0$ depending only on $N,p,\mu$ such that
$||u||_{L^{p*}(D_{R})}\le CR^{\tau_{0}}$ if $0<R\leq\rho/8$ and
$||u||_{L^{p*}(D_{R})}\le CR^{-\tau_{1}}$ if $R\ge8/\rho$.
Therefore, by letting $r_0=\rho/8, r_1=8/\rho$ and  inserting the estimates of $||u||_{L^{p*}(D_{R})}$
into (\ref{eq:maximum estimate by integral}), we complete the
proof of Proposition \ref{thm:poor estimate}. \end{proof}

Next  we prove that
some special functions are supersolutions.
\begin{proposition}\label{prop:computation of sub-super solution}

(1) Given two constants $A,\alpha\in \mathbb{R}$,  $A>0$ and
 $\alpha<p$. There exist constants $1>\delta,\epsilon>0$ depending on $N,p,\mu,A,\alpha
$
 such that the function $v(x)=|x|^{-\gamma_1}(1-\delta|x|^{\epsilon})\in \mathcal{D}^{1,p}(B_{R_0})$ is a positive supersolution to equation
\begin{eqnarray}
-L_{p}v= g(x) {|v|^{p-2}v},&&  x\in
B_{R_0},\label{eq: 2.4.1}\end{eqnarray} where $g$ is a positive function in  $L^{\frac Np}(B_{R_0})$ satisfying
$$ g(x)\ge A|x|^{-\alpha },\quad x\in B_{R_0},$$
with some constant $1>R_0=R_0(N,p,\mu,A,\alpha)>0$.

(2) Given two constants $A,\alpha\in \mathbb{R}$, $A>0$
and
 $\alpha>p$. There exist constants $1>\delta,\epsilon>0$ depending on $N,p,\mu,A,\alpha
$ such that the function
$v(x)=|x|^{-\gamma_2}(1-\delta|x|^{-\epsilon})\in \mathcal{D}^{1,p}(B^c_{R_1})$ is a supersolution
to  equation
\begin{eqnarray}
-L_{p}v=g(x) {|v|^{p-2}v},&& x\in B^c_{R_1},\label{eq: 2.4.2}\end{eqnarray}
where $g$ is a positive function in $L^{\frac Np}(B^c_{R_1})$  satisfying
$$ g(x)\ge A|x|^{-\alpha },\quad x\in B^c_{R_1},$$
with some constant $R_1=R_1(N,p,\mu,A,\alpha)>1$.
 \end{proposition}
 \begin{proof}
Let $\gamma, \delta,\epsilon\in \mathbb{R}$ and define function $v(x)=|x|^{-\gamma}(1-\delta|x|^{\epsilon})$.
Direct computation shows that
\begin{eqnarray*}
-L_{p}v=\frac
{h(\delta|x|^{\epsilon})}{\left|1-\delta|x|^{\epsilon}\right|^{p-2}\left(1-\delta|x|^{\epsilon}\right)|x|^{p}}
{|v|^{p-2}v},&& \hbox{for}\, x\ne 0,\end{eqnarray*} where
  \[
h(t)\equiv|\gamma-(\gamma-\epsilon)t|^{p-2}[k(\gamma-\epsilon)t-k(\gamma)]-\mu
|1-t|^{p-2}(1-t),\quad t\in[0,\infty),\] and
$k(t)\equiv(p-1)t^{2}-(N-p)t, t\ge 0$. It is easy to prove that $v$ is a weak solution of equation  (\ref{eq: 2.4.1}) and equation (\ref{eq: 2.4.2}) with $$g(x)=\frac
{h(\delta|x|^{\epsilon})}{\left|1-\delta|x|^{\epsilon}\right|^{p-2}
\left(1-\delta|x|^{\epsilon}\right)|x|^{p}}
.$$
It remains to prove  that $g$ satisfies those properties mentioned in the proposition.

Note that $h(0)=-|\gamma|^{p-2}k(\gamma)-\mu$, where
$k(\gamma)=(p-1)\gamma^{2}-(N-p)\gamma$. Thus by the definition of
$\gamma_1$ and $\gamma_2$, as in (\ref{eq: gamma 1 and gamma 2}),
we have $h(0)=0$ when $\gamma=\gamma_1$ or $\gamma=\gamma_2$.

Also we have
$$h'(0)=(p-1)\gamma^{p-2}\left(-p\gamma+N-p+\epsilon \right)\epsilon.$$
So \begin{equation*}h'(0)>0\quad\hbox{if}\,
\gamma=\gamma_1,\epsilon>0\,\,\hbox{or}\,\,
\gamma=\gamma_2,\epsilon<0.\end{equation*} Therefore, there exists
$1>\delta_h>0$ such that $2h'(0)t\ge h(t)\ge \frac 12 h'(0)t>0$
for $t\in(0,\delta_h)$.

To obtain (1), we choose  $\delta=\min\{\delta_h,1/2\}, \epsilon=(p-\alpha)/2>0$ and
$$R_0=\min\left\{1,\left(\frac
{\delta}{2A} h'(0)\right)^{1/(p-\alpha-\epsilon)}\right\}.$$
Then  function $v(x)=|x|^{-\gamma_1}(1-\delta|x|^{\epsilon})$ is positive in $\mathcal{D}^{1,p}(B_{R_0})$ since $\gamma_1<(N-p)/p, \delta>0, \epsilon>0$  and function $g$ given above satisfies
$$ A|x|^{-\alpha}\le\frac
{\delta}2 h'(0)|x|^{\epsilon-p}\le g(x) \le 2^p {\delta}
h'(0)|x|^{\epsilon-p}, \qquad\forall\,x\in B_{R_0}.$$ The last
 inequality implies that $g\in L^{\frac Np}(B_{R_0})$.

(2) is obtained similarly.
 \end{proof}

\section{Comparison principle}
\begin{comment}To study the simplicity of the first eigenvalue of $p-$Laplacian
operator on an arbitrary bounded domain $\Omega$, i.e.,
\begin{eqnarray*}-\Delta_p u=\lambda_1 |u|^{p-2}u,&&x\in \Omega,\end{eqnarray*} where $\lambda_1>0$ is defined as the minimum of the Rayleigh
quotients: $\lambda_1=\inf\{||\nabla
v||_{p,\Omega}^p/||v||^p_{p,\Omega}:v\in \mathcal{D}^{1,p}_0(\Omega), v\ne
0\}$, Lindqvist \cite{Lin} observed that \cite{A,DS} supplied
a good idea on the choice of fine test functions for weak
solutions of the equation above. In this paper we shall also choose similar fine test functions  as that of \cite{Lin}.\end{comment}

In this section, we prove the comparison principle for subsolutions and supersolutions to equation (\ref{eq: general equation}). We start with the following pointwise estimate.

\begin{lemma}
\label{lem:auxilury lemma}   For all weakly differentiable
positive functions $u,v$ on a domain $\Omega$, we have for $p\geq2$,
\[|\nabla u|^{p-2}\nabla u\cdot\nabla \left(u-\frac{v^{p}}{u^p}u\right)
+|\nabla v|^{p-2}\nabla v\cdot\nabla\left(v-\frac{u^{p}}{v^p}v\right)\ge C_{p}(u^{p}+v^{p})|\nabla\log
u-\nabla\log v|^{p};\]
 and also for $1<p<2$,
\[|\nabla u|^{p-2}\nabla u\cdot\nabla \left(u-\frac{v^{p}}{u^p}u\right)
+|\nabla v|^{p-2}\nabla v\cdot\nabla\left(v-\frac{u^{p}}{v^p}v\right)\ge C_{p}(u^{p}+v^{p})\frac{|\nabla\log
u-\nabla\log v|^{2}}{(|\nabla\log u|+|\nabla\log v|)^{2-p}},
\] where $C_p$ are positive constants depending only on $p$.
\end{lemma}
\begin{proof}Let $u,v$ be two weakly differentiable positive
functions, and
$\eta_{1}=u-\frac{v^{p}}{u^p}u$, $\eta_{2}=v-\frac{u^{p}}{v^p}v$.
Then
\[
 |\nabla u|^{p-2}\nabla u\cdot \nabla\eta_{1}= |\nabla u|^{p}-v^{p}\left(|\nabla\log u|^{p}
 +p|\nabla\log u|^{p-2}\nabla\log u\cdot(\nabla\log v-\nabla\log
 u)\right),
\]
and
\[
|\nabla v|^{p-2}\nabla v\cdot\nabla\eta_{2}=|\nabla v|^{p}
-u^{p}\left(|\nabla\log v|^{p}+p|\nabla\log v|^{p-2}\nabla\log
v\cdot(\nabla\log u-\nabla\log v)\right).
 \]

When $p\geq2$,   applying the following elementary inequality (see
\cite{Lin}):
\begin{eqnarray*}
|a|^{p}\geq|b|^{p}+p|b|^{p-2}b\cdot(a-b)+\frac{|a-b|^{p}}{2^{p-1}-1},&&\forall\,
a,b\in\mathbb{R}^{N},
\end{eqnarray*}
 we obtain that,
\begin{eqnarray*}
|\nabla u|^{p-2}\nabla u\cdot\nabla\eta_{1}
  & \geq & |\nabla u|^{p}-v^{p}(|\nabla\log v|^{p}-C_{p}|\nabla\log u-\nabla\log
 v|^{p})\\
 & = & |\nabla u|^{p}-|\nabla v|^{p}+C_{p}v^{p}|\nabla\log u-\nabla\log
 v|^{p},
\end{eqnarray*}
and that
\[
|\nabla v|^{p-2}\nabla v\cdot\nabla\eta_{2}\geq |\nabla
v|^{p}-|\nabla u|^{p}+C_{p}u^{p}|\nabla\log u-\nabla\log v|^{p}.
\]
Adding these two inequalities together,  we finish the proof of the lemma in the case $p\ge 2$.

The case $1<p<2$ follows from the following elementary inequality (see
\cite{Lin}):
\begin{eqnarray*}
|a|^{p}\geq|b|^{p}+p|b|^{p-2}b\cdot(a-b)+C_{p}\frac{|a-b|^{2}}{(|a|+|b|)^{2-p}},&&\forall\,
a,b\in\mathbb{R}^{N},
\end{eqnarray*}
where $C_{p}$ is a positive constant depending only on $p$.
\end{proof}

In the following two theorems, Theorem \ref{thm: Comparison}
 and Theorem \ref{thm: Comparison at infinity},  we prove the comparison principle. Theorem \ref{thm: Comparison} deals with the bounded domains, and Theorem \ref{thm: Comparison at infinity} with the exterior domains.
 \begin{theorem}
\label{thm: Comparison} Let $\Omega$ be a bounded domain and $f\in L^{\frac{N}{p}}(\Omega)$. Let  $u\in \mathcal{D}^{1,p}(\Omega)$ be a weak subsolution of equation (\ref{eq: general equation})
%\begin{eqnarray} -L_{p}u= f|u|^{p-2}u &  &\hbox{in}\:\Omega,\label{eq: subsolution} \end{eqnarray}
and $v\in \mathcal{D}^{1,p}(\Omega)$  a weak supersolution of
\begin{eqnarray}
-L_{p}v= g|v|^{p-2}v &  &
\hbox{in}\:\Omega\label{eq: supersolution}
\end{eqnarray} such that  $\inf_{\Omega}v>0$,
where  $g$ is a given function in $L^{\frac{N}{p}}(\Omega)$ such that $f\leq g$  in $\Omega$.  If  $u\leq v$ on $\partial\Omega$, then
\begin{eqnarray*}
u\leq v && \hbox{in}\:\Omega.
\end{eqnarray*}
\end{theorem}
\begin{proof}
We only prove Theorem \ref{thm: Comparison} in the case when $p\geq 2$. We can prove similarly Theorem \ref{thm: Comparison} in the case when $1<p<2$.
For  $m>1$, it is easy to check that functions $\eta_{1}=u^{1-p}\min{((u^{p}-v^{p})^{+},{m})}$ and
$\eta_{2} =-v^{1-p}\min{((u^{p}-v^{p})^{+},{m})}$ can be taken as test functions
of equations (\ref{eq: general equation}) and equation (\ref{eq: supersolution}) respectively.
 Therefore, substituting $\eta_{1}$ into equation (\ref{eq: general equation}) and $\eta_{2} $ into
 equation (\ref{eq: supersolution}),  and then adding together, we obtain that
%Direct computation shows that $\eta_1,\eta_2\in
%\mathcal{D}^{1,p}_0{(\Omega)}$ (see the definition in Remark \ref{rmk: 1}). %,
\begin{eqnarray*}
  \langle-L_{p}u,\eta_{1}\rangle+\langle-L_{p}v,\eta_{2}\rangle
 & \leq & \int_{\Omega}f|u|^{p-2}u\eta_{1}+\int_{\Omega}g|v|^{p-2}v\eta_{2}\\
  & = & \int_{{\{u^{p}-v^{p}\geq m\}}}m(f-g)+\int_{{\{0\leq u^{p}-v^{p}\leq m\}}}(f-g)(u^{p}-v^{p})\\
 & \leq & 0  \qquad\qquad\hbox{since } f\le g,
\end{eqnarray*}
where $\langle-L_{p}w,\eta\rangle $ is defined as
$$\langle-L_{p}w,\eta\rangle =\int_{\Omega}\left(|\nabla w|^{p-2}\nabla
w \cdot\nabla \eta-\frac{\mu}{|x|^{p}}|w|^{p-2}w\eta\right)$$ for
all $w,\eta\in \mathcal{D}^{1,p}_0(\Omega) $.

By the definition of $\eta_1,\eta_2$, we obtain that
\begin{eqnarray*}
\langle-L_{p}u,\eta_{1}\rangle & = & \int_{\{u^{p}-v^{p}\geq
m\}}\left(|\nabla u|^{p-2}\nabla u
\cdot\nabla(mu^{1-p})-\mu\frac{u^{p-1}}{|x|^{p}}mu^{1-p}\right)\\
  &  & +\int_{\{0\leq u^{p}-v^{p}\leq m\}}\left(|\nabla u|^{p-2}\nabla  u\cdot\nabla
  \left(u-\frac{v^{p}}{u^{p}}u\right)-\mu\frac{u^{p-1}}{|x|^{p}}\left(u-\frac{v^{p}}{u^{p}}u\right)\right),
\end{eqnarray*}
and  \begin{eqnarray*} \langle-L_{p}v,\eta_{2}\rangle & = &
\int_{\{u^{p}-v^{p}\geq m\}}\left(|\nabla v|^{p-2}\nabla
v\cdot\nabla(-mv^{1-p})+\mu\frac{v^{p-1}}{|x|^{p}}mv^{1-p}\right)\\
 &  & +\int_{\{0\leq u^{p}-v^{p}\leq m\}}\left(|\nabla v|^{p-2}\nabla
 v\cdot\nabla\left(v-\frac{u^{p}}{v^{p}}v\right)-\mu\frac{v^{p-1}}{|x|^{p}}\left(v-\frac{u^{p}}{v^{p}}v\right)\right)
\end{eqnarray*} respectively.  Hence we have
\begin{eqnarray*}
\langle-L_{p}u,\eta_{1}\rangle
&= & \int_{\{u^{p}-v^{p}\geq m\}}\left(m(1-p)u^{-p}|\nabla u|^{p}-m\mu |x|^{-p}\right)\\
 &  & +\int_{\{0\leq u^{p}-v^{p}\leq m\}}\left(|\nabla u|^{p-2}\nabla
 u\cdot\nabla\left(u-\frac{v^{p}}{u^{p}}u\right)-\mu\frac{u^{p}-v^p}{|x|^{p}}\right),
\end{eqnarray*} and
\[
\langle-L_{p}v,\eta_{2}\rangle
 \geq  \int_{\{u^{p}-v^{p}\geq m\}}m\mu |x|^{-p}
  +\int_{\{0\leq u^{p}-v^{p}\leq m\}}\left(|\nabla v|^{p-2}\nabla
 v\cdot\nabla\left(v-\frac{u^{p}}{v^{p}}v\right)-\mu\frac{v^{p}-u^p}{|x|^{p}}\right),
\]
since  $|\nabla v|^{p-2}\nabla
v\cdot\nabla(-mv^{1-p})\ge 0$.
Therefore we obtain that
\begin{eqnarray*}
 \langle-L_{p}u,\eta_{1}\rangle+\langle-L_{p}v,\eta_{2}\rangle & \geq & \int_{\{u^{p}-v^{p}\geq m\}}m(1-p)u^{-p}|\nabla u|^{p}\\
  &  & +\int_{\{0\leq u^{p}-v^{p}\leq m\}}\left(|\nabla u|^{p-2}\nabla u\cdot\nabla\left(u-\frac{v^{p}}
 {u^{p}}u\right)+|\nabla v|^{p-2}\nabla v\cdot\nabla\left(v-\frac{u^{p}}{v^{p}}v\right)\right).
 \end{eqnarray*}
 Estimate the right hand side of the above equation as follows: for
the first term we have
 \[
  (p-1)\int_{{\{u^{p}-v^{p}\geq m\}}}mu^{-p}|\nabla u|^{p}
\leq  (p-1)\int_{{\{u^{p}\geq m\}}}|\nabla u|^{p}\rightarrow  0\quad\hbox{as}\:
 m\rightarrow\infty,
\]
since $\lim_{m\to \infty}|\{u^{p}\geq m\}|=0$,  and  for  the second term we  apply Lemma
\ref{lem:auxilury lemma} to obtain that
\[
|\nabla u|^{p-2}\nabla u\cdot\nabla\left(u-\frac{v^{p}}{u^{p}}u\right)+|\nabla v|^{p-2}
\nabla v\cdot\nabla\left(v-\frac{u^{p}}{v^{p}}v\right) \ge C_{p} (u^{p}+v^{p})|\nabla\log u-\nabla\log v|^{p}
\] for some positive constant $C_p$ depending only on $p$.
Recall that $\langle-L_{p}u,\eta_{1}\rangle+\langle-L_{p}v,\eta_{2}\rangle\le 0$.
%\begin{eqnarray*}
%0 & \geq & \int_{{\{u^{p}-v^{p}\geq m\}}}m|\nabla u|^{p-2}\nabla u\nabla(u^{1-p})\\
 %&  & +C_{p}\int_{{\{0\leq u^{p}-v^{p}\leq m\}}}(u^{p}+v^{p})|\nabla\log u-\nabla\log v|^{p}.\\
%\end{eqnarray*}
Letting $m\rightarrow\infty$,  we obtain that
\begin{eqnarray*}
\int_{{\{0\leq u^{p}-v^{p}\}}}(u^{p}+v^{p})|\nabla\log
u-\nabla\log v|^{p}=0,
\end{eqnarray*}
which implies that
\[\log u=\log v+C\quad \hbox{on}\: \{x\in \Omega; u(x)\ge v(x)\},\]
i.e.
\[ u=Cv\quad \hbox{on}\, \{x\in \Omega; u(x)\ge v(x)\},\]
for some positive constant $C>0$.
Since we assume in the theorem that $\inf_{\Omega}v>0$, it follows that $C=1$. This implies that
\[u\le v,\quad \hbox{in}\: \Omega.\]
This finishes the proof of Theorem \ref{thm: Comparison} in the case $p\ge 2$.
 \end{proof}

Corresponding  comparison principle in exterior domains is given in the
following theorem.

\begin{theorem}
\label{thm: Comparison at infinity} Let $\Omega$ be an exterior
domain such that $\Omega^c=\mathbb{R}^N\backslash \Omega $ is
bounded and $f\in L^{\frac Np}(\Omega)$. Let $u\in\mathcal{D}^{1,p}(\Omega)$ be a subsolution of equation \eqref{eq: general equation}
%\begin{eqnarray} -L_{p}u=f|u|^{p-2}u &  & \hbox{in}\: \Omega,\label{eq: subsolution 2} \end{eqnarray}
and $v\in\mathcal{D}^{1,p}(\Omega)$  a positive supersolution of
\begin{eqnarray}
-L_{p}v= g|v|^{p-2}v &  & \hbox{in}\: \Omega,\label{eq:
supersolution 2}
\end{eqnarray}
such that $\inf_{\pa\Om}v>0$, where functions  $g$ belongs to  $ L^{\frac{N}{p}}(\Omega)$ and  $f\leq g$ in $\Omega$.   Moreover, assume that
\begin{equation}\limsup_{R\to\infty}
\frac{1}{R}\int_{B_{2R}\backslash B_{R}}u^{p}|\nabla\log
v|^{p-1}=0. \label{eq: special condition on v}\end{equation}
If $u\leq
v$ on $\partial \Omega$, then
\begin{eqnarray*}
u\leq v &  & \hbox{in}\: \Omega.
\end{eqnarray*}
\end{theorem}
\begin{proof}
 We only prove Theorem \ref{thm: Comparison at infinity} in  the case when $p\ge
 2$. We can prove similarly Theorem \ref{thm: Comparison at infinity} in  the case when
 $1< p< 2$.
Fix $R>2\hbox{diam}(\Omega^c)$ and $m\ge 1$. Let $\eta\in
C_{0}^{\infty}(B_{2R})$ be a cut-off function  such that $0\leq\eta\leq1$, $\eta\equiv 1$
on $B_{R}$ and $|\nabla\eta|\le\frac{2}{R}$. Substituting  test functions
\[ \varphi_{1}=\eta
u^{1-p}\min{((u^{p}-v^{p})^{+},m)}\quad\hbox{and}\quad\varphi_{2}=-\eta
v^{1-p}\min{((u^{p}-v^{p})^{+},m)} \] into equation (\ref{eq: general equation}) and  equation (\ref{eq: supersolution 2})
respectively, and  adding together,  we have
\begin{eqnarray*}
  \langle-L_{p}u,\varphi_{1}\rangle+\langle-L_{p}v,\varphi_{2}\rangle
 & \leq & \int_{\Omega}f|u|^{p-2}u\varphi_{1}+\int_{\Omega}g|v|^{p-2}v\varphi_{2}\\
  & = & \int_{{\{u^{p}-v^{p}\geq m\}}}m(f-g)\eta+\int_{{\{0\leq u^{p}-v^{p}\leq m\}}}(f-g)(u^{p}-v^{p})\eta\\
 & \leq & 0  \qquad\qquad\hbox{since}\: f\le g.
\end{eqnarray*}
On the other hand, by the definition of $\varphi_1$ and
$\varphi_2$, we have
\begin{eqnarray*}
 &  & \langle-L_{p}u,\varphi_{1}\rangle+\langle-L_{p}v,\varphi_{2}\rangle\\
 & = & \int_{\{0\le u^{p}-v^{p}\le m\}}\left(\eta|\nabla u|^{p-2}\nabla u\cdot\nabla\left(u-\frac{v^{p}}{u^{p}}u\right)+
 \eta|\nabla v|^{p-2}\nabla v\cdot \nabla\left(v-\frac{u^{p}}{v^{p}}v\right)\right)\\
 &  & +\int_{\{0\le u^{p}-v^{p}\le m\}}\left(\left(u-\frac{v^{p}}{u^{p}}u\right)|\nabla u|^{p-2}
 \nabla u\cdot \nabla\eta+\left(v-\frac{u^{p}}{v^{p}}v\right)
 |\nabla v|^{p-2}\nabla v\cdot \nabla\eta\right)\\
 &  & +\int_{\{u^{p}-v^{p}\ge m\}}|\nabla u|^{p-2}\nabla u\cdot \nabla(m\eta u^{1-p})+\int_{\{u^{p}-v^{p}\ge m\}}|\nabla v|^{p-2}\nabla v\cdot \nabla(-m\eta v^{1-p})\\
 & =: & I_{1}+I_{2}+I_{3}+I_{4}.
\end{eqnarray*}
 Lemma \ref{lem:auxilury lemma} implies that
\begin{eqnarray*}
I_{1} & \geq & C_p\int_{\{0\le u^{p}-v^{p}\le m\}}\eta(u^{p}+v^{p})|\nabla\log u-\nabla\log v|^{p}\\
 & \geq & C_p\int_{\{0\le u^{p}-v^{p}\le m\}\cap B_{R}}(u^{p}+v^{p})|\nabla\log u-\nabla\log v|^{p},
\end{eqnarray*}
where $C_p>0$ is independent of $m,R$.

We estimate $I_{k}$ $(k=2,3,4)$ as follows. For $I_2$, we have
\begin{eqnarray*}
|I_{2}|  &\leq & \int_{\{0\le u^{p}-v^{p}\}}\left(|\nabla
u|^{p-1}|u|+|\nabla v|^{p-1}v\right)|\nabla\eta|
+\int_{\{0\le u^{p}-v^{p}\}}u^{p}|\nabla\log v|^{p-1}|\nabla\eta|    \nonumber\\
&\leq& \frac{2}{R}\int_{B_{2R}\backslash B_{R}}\left(|\nabla
u|^{p-1}|u|+|\nabla v|^{p-1}v\right) +
\frac{2}{R}\int_{B_{2R}\backslash B_{R}}u^{p}|\nabla\log v|^{p-1}
 =:  J_{1}+J_{2}.\label{eq:J2}
\end{eqnarray*}
 H\"older's inequality implies that
\begin{eqnarray*}
 %& \le & \frac{C}{R}\int_{B_{2R}\backslash B_{R}}\left(|\nabla u|^{p-1}|u|+|\nabla v|^{p-1}v\right)\\
J_{1} & \leq & C\left(\int_{B_{2R}\backslash B_{R}}|\nabla u|^{p}\right)^{\frac{p-1}{p}}
\left(\int_{B_{2R}\backslash B_{R}}|u|^{p^{*}}\right)^{1/p^{*}}+C\left(\int_{B_{2R}
\backslash B_{R}}|\nabla v|^{p}\right)^{\frac{p-1}{p}}\left(\int_{B_{2R}\backslash B_{R}}|v|^{p^{*}}\right)^{1/p^{*}}\\
 & = & o(1)\quad\quad as\: R\rightarrow\infty,
\end{eqnarray*}
 and by assumption \eqref{eq: special condition on v}
\begin{eqnarray*}
 J_{2}\to0 & & \hbox{as}\: R\to\infty.
\end{eqnarray*}
Note that $J_1,J_2$ are independent of $m$, so
\[
\lim_{m,R\to\infty}I_{2}=0.
\]
For $I_{3}$, it is easy to see that
\[
I_{3} =m\int_{\{u^{p}-v^{p}\ge m\}}\left(|\nabla u|^{p-2}\nabla
u\cdot\nabla\eta u^{1-p}+(1-p)\eta u^{-p}|\nabla u|^{p}\right).
\]
 So
\[
|I_{3}| \le \int_{\{u^{p}\ge m\}}\left(|\nabla
u|^{p-1}|\nabla\eta|u+(p-1)\eta|\nabla u|^{p}\right)\le
J_1+(p-1)\int_{\{u^{p}\ge m\}}|\nabla u|^{p}.
\]
Since the level set $\{u^{p}\ge m\}$  vanishes as $m\to\infty$,
one deduces that
\[
\lim_{m,R\to\infty}I_{3}=0.
\]
For the last term $I_{4}$, we have
\begin{eqnarray*}
I_{4} & = & \int_{\{u^{p}-v^{p}\ge m\}}\left(m(p-1)\eta v^{-p}|\nabla v|^{p}-mv^{1-p}|\nabla v|^{p-2}\nabla v\cdot\nabla\eta\right)\\
 & \ge & -m\int_{\{u^{p}-v^{p}\ge m\}}|\nabla\log v|^{p-1}|\nabla\eta|\\
 & \ge & -\int_{\{u^{p}\ge m\}}|\nabla\log
 v|^{p-1}|\nabla\eta|u^{p}\ge -J_2,
\end{eqnarray*}
 which converges to zero as $m,R\to\infty$ by assumption (\ref{eq: special condition on v}). Hence combining
together all above estimates we obtain that
\[
0\ge \limsup_{R,m\to\infty}\left(\langle-L_{p}u,\varphi_{1}\rangle+\langle-L_{p}v,\varphi_{2}\rangle\right)\ge
C\int_{\{u\ge v\}}(u^{p}+v^{p})|\nabla\log u-\nabla\log v|^{p}.
\]
Thus
\[
\int_{\{u\ge v\}}(u^{p}+v^{p})|\nabla\log u-\nabla\log v|^{p}=0,
\]
 which implies that $u\le  v$ in $\Omega$ as proved before. This finishes the
proof.
\end{proof}

\begin{comment}  We give two examples in which condition (\ref{eq: special condition on v}) is satisfied.

\textbf{Example 1.} If the supersolution $v$ is such that $|\nabla
\log v|$ is a bounded and the subsolution $u$ satisfies
\begin{equation}R^{p-1}\left(\int_{B_{2R}\backslash
B_{R}}|u|^{p^{*}}\right)^{p/p^{*}}\to 0\qquad \hbox{as}\:R\to
\infty,\label{rmk: 4.3.1}\end{equation}then condition (\ref{eq:
special condition on v}) is satisfied by applying H\"older's
inequality.

\textbf{Example 2.} If the supersolution $v$ is such that
 $$\limsup_{|x|\to \infty}|\nabla \log v(x)||x|<\infty,$$
 then H\"older's inequality and the fact that subsolution
  $u$ belongs to $ L^{p^*}(\Omega)$ imply that (\ref{eq: special condition on v}) is true.
\end{comment}

\section{Proofs of main results}
In the following we will  prove Theorem \ref{thm:given growth comparison at origin 1}, Theorem \ref{thm:given growth comparison at origin 2}, Theorem \ref{thm:given growth comparison at infinity 1} and Theorem \ref{thm:given growth comparison at infinity 2} first, and then we prove Theorem \ref{thm:main result} and Theorem \ref{thm: inverse of main resuls}.
 \begin{proof}[Proof of Theorem \ref{thm:given growth comparison at
origin 1}]  Let  $u\in \mathcal{D}^{1,p}(\Omega)$ be a
subsolution to equation (\ref{eq: general equation})  with $f\in L^{\frac Np}(\Omega)$
satisfying $f(x)\leq A|x|^{-\alpha}$ in $\Omega$ for some constants
$A>0,p>\alpha$.

By Proposition \ref{prop:computation of
sub-super solution},  the function $v(x)=|x|^{-\gamma_1}(1-\delta|x|^{\epsilon})$ is a positive supersolution  of equation
(\ref{eq: supersolution})  in $B_{R_0}\subset\Omega$  with
function $g\in L^{\frac Np}(B_{R_0})$ satisfying $g(x)\ge
A|x|^{-\alpha }$ for all $x\in B_{R_0}$, where $1>\delta,\epsilon,
R_0>0$ are constants depending on $N,p,\mu,A,\alpha$.  Obviously we have $g\ge f$ in $B_{R_0}$.

 Let $k
>0$ and define function $\tilde{v}=C(M+k)v$, where $C=\sup_{\partial B_{R_0}}v^{-1},M=\sup_{\partial B_{R_0}}u^+$. Then $\tilde{v}$ is also a positive supersolution to equation
\eqref{eq: supersolution} in the ball $B_{R_0}$ with the same function $g$
as above. Moreover,  $\inf_{B_{R_0}}\tilde{v}=M+k>0$ and $u\le
\tilde{v}$ on $\partial B_{R_0}$. Thus we can  apply  Theorem
\ref{thm: Comparison} to the subsolution $u$ of equation
\eqref{eq: general equation} and the supersolution $\tilde{v}$ of
equation (\ref{eq: supersolution})  on the ball $B_{R_0}$ to
conclude  that
\begin{eqnarray*}u(x)\le \tilde{v}(x) \leq C\Big(\sup_{\partial B_{R_0}}u^+
+k\Big)|x|^{-\gamma_{1}},&&\hbox{for}\: x\in B_{R_0}
\end{eqnarray*}
 with  constant  $C=\sup_{\partial B_{R_0}}v^{-1}$ independent of $k$. Now Theorem \ref{thm:given growth comparison at
origin 1} follows by taking $k\to 0$.\end{proof}

\begin{proof}[Proof of Theorem \ref{thm:given growth comparison at
origin 2}] Suppose that $ u\in \mathcal{D}^{1,p}(\Omega)$ is a
nonnegative supersolution to equation (\ref{eq: general equation})
with $f\ge 0$. Then $u$ is also a  supersolution to equation
(\ref{eq: supersolution}) with $g= 0$. To prove the theorem, we
may assume that $m\equiv \inf_{\Omega}u>0$, otherwise  Theorem
\ref{thm:given growth comparison at origin 2} is trivial since we
assume that $u\ge 0$.

Define function
$v(x)=Cm {|x|}^{-\gamma_1}$ in $\Omega$, where
$C=\inf_{\partial\Omega}|x|^{\gamma_1}$. Then $v$ is a  solution of
equation (\ref{eq:  general equation}) with $f\equiv 0$,  thus
 a subsolution of equation (\ref{eq:  general equation}) with $f\equiv 0$. Obviously
there holds $u\ge v$ on $\partial \Omega$. So applying Theorem
\ref{thm: Comparison} to the subsolution $v$ of
equation (\ref{eq:  general equation}) and the supersolution
$u$ of
equation (\ref{eq: supersolution}) on the domain $\Omega$,  we conclude that
\begin{eqnarray*}u(x)\ge v(x)=Cm {|x|}^{-\gamma_1}&&\hbox{for}\: x\in \Omega. \end{eqnarray*}
This finishes the proof of Theorem \ref{thm:given growth comparison at
origin 2}.\end{proof}

\begin{proof}[Proof of Theorem \ref{thm:given growth comparison at
infinity 1}]  Let  $u\in \mathcal{D}^{1,p}(\Omega)$ be a
subsolution to equation (\ref{eq: general equation}) with $f\in L^{\frac
Np}(\Omega)$ satisfying $f(x)\leq A|x|^{-\alpha}$ in $\Omega$ for some
constants $A>0,\alpha>p$.

By Proposition \ref{prop:computation of sub-super solution}, the function $v(x)=|x|^{-\gamma_2}(1-\delta|x|^{-\epsilon})$ is a positive supersolution  of equation
(\ref{eq: supersolution 2})  in $B_{R_1}^c$  with function $g\in
L^{\frac Np}(B_{R_1}^c)$ satisfying $g(x)\ge A|x|^{-\alpha }\ge
f(x)$ for all $x\in B_{R_1}^c$, where
 $1>\delta,\epsilon>0, R_1>1$ are constants  depending on
$N,p,\mu,A,\alpha$.

 Let $k
>0$ and define $\tilde{v}=C(M+k)v$, where $C=\sup_{\partial B_{R_1}^c}v^{-1},M=\sup_{\partial B_{R_1}^c}u^+$. Then $\tilde{v}$ is a positive supersolution to
\eqref{eq: supersolution 2} in  $B_{R_1}^c$ with the same function $g$
given above. Moreover,  $\inf_{\pa B_{R_1}^c}\tilde{v}=M+k>0$ and
$u\le \tilde{v}$ on $\partial B_{R_1}^c$. We verify the condition
(\ref{eq: special condition on v}) as follows: by H\"older's
inequality
\[
\frac{1}{R}\int_{B_{2R}\backslash B_{R}}u^{p}|\nabla\log
v|^{p-1}\le \frac C{R^{p}}\int_{B_{2R}\backslash B_{R}}u^{p}\leq
C\left(\int_{B_{2R}\backslash B_{R}}|u|^{p^{*}}\right)^{1/p^{*}},
\]
where $C$ is a constant independent of $R$. The first
inequality follows by noting that $|\nabla \log v(x)|\le
C|x|^{-1}$ with some constant  $C$  independent of $R$. Therefore
(\ref{eq: special condition on v}) holds since $u\in L^{p^*}(\R^N)$.

 Thus we can  apply  Theorem \ref{thm: Comparison at infinity} to the
subsolution $u$ of equation \eqref{eq:  general equation} and the
supersolution $\tilde{v}$ of equation (\ref{eq: supersolution 2})
in $B_{R_1}^c$ to conclude  that
\begin{eqnarray*}u(x)\le \tilde{v}(x) \le  C\Big(\sup_{\partial B_{R_1}}u^+
+k\Big)|x|^{-\gamma_{2}}&&\hbox{for}\: x\in B_{R_1}^c,
\end{eqnarray*}
 with  constant  $C=\sup_{\partial B_{R_1}^c}v^{-1}$  independent of $k$. Now the theorem follows by taking $k\to 0$. \end{proof}

\begin{proof}[Proof of Theorem \ref{thm:given growth comparison at
infinity 2}] Suppose that  $ u\in \mathcal{D}^{1,p}(\Omega)$ is a
nonnegative supersolution to (\ref{eq: general equation}) with
$f\ge 0$, then $u$ is a nonnegative supersolution to equation
\eqref{eq: supersolution 2} with $g= 0$. We may assume that
$m\equiv \inf_{\partial \Omega}u>0$, otherwise the theorem is
trivial  since we assume that $u\ge 0$. The positivity of $u$ in
$\Omega$ is a consequence of the fact   that $u$ is also a
nonnegative supersolution to $p-$Laplacian equation
$$-\Delta_p u\ge   0$$
 in $\Omega$ since $-\Delta_p u\ge \mu{u^{p-1}}/{|x|^p}\ge 0$.  Moreover, it is well known (see
\cite{Lindqvist2}) that
\begin{equation}\int_{B_{2R}\backslash B_R}|\nabla \log
u|^p\le CR^{N-p}\label{eq: log growth},\end{equation} for all
$R$ large enough and $C>0$  a constant independent of $R$.

Let $C_1=\inf_{x\in\partial \Omega}|x|^{\gamma_2}$. The function
$v=C_1 m|x|^{-\gamma_2}$ is a solution to (\ref{eq:  general
equation}) with $f=0$, and thus a subsolution to  (\ref{eq:
general equation}) with $f=0$. Condition (\ref{eq: special
condition on v}) on $u$ and $v$ is also satisfied by (\ref{eq: log
growth}) and H\"older's inequality:
\begin{eqnarray*}
R^{-1}\int_{B_{2R}\backslash B_{R}}v^{p}|\nabla\log u|^{p-1}&\le
& CR^{-1-p\gamma_2}\left(\int_{B_{2R}\backslash B_{R}}|\nabla\log
u|^{p}\right)^{\frac {p-1}p}|B_{2R}(0)\backslash B_R(0)|^{\frac 1p}\\
&\leq& CR^{-1-p\gamma_2+\frac {p-1}p (N-p)+\frac
Np}\\ &\to& 0\quad\text{as} \,R\to \infty,
\end{eqnarray*}
since $-1-p\gamma_2+{\frac {p-1}p}(N-p)+\frac Np<0$ due to the
fact that $\gamma_2>\frac {N-p}p$.

Thus we can  apply  Theorem \ref{thm: Comparison at infinity} to
the supersolution $u$ of equation (\ref{eq: supersolution 2}) and
the subsolution $v$ of equation \eqref{eq:  general equation} in
$B_{R_1}^c$ to conclude  that \begin{eqnarray*}u(x)\ge v(x) =
C_1 m|x|^{-\gamma_{2}}&&\hbox{for}\: x\in B_{R_1}^c.
\end{eqnarray*}
This completes the proof of Theorem \ref{thm:given growth comparison at
infinity 2}.\end{proof}

Now we  prove Theorem \ref{thm:main result} and Theorem \ref{thm: inverse of main resuls}.

\begin{proof}[Proof of Theorem
 \ref{thm:main result}]
 Let $u$  be a weak  solution of (\ref{eq:object}).
By Proposition \ref{thm:poor estimate}, there exists a positive constant $C $ depending on $N,p,\mu,||Q||_{\infty}$ and $u$ such that
\begin{eqnarray*}
|u(x)|\leq {C}{|x|^{-\left(\frac{N-p}{p}-\tau_0\right)}} &  &
\hbox{for}\:|x|< r_0,\end{eqnarray*} and that
\begin{eqnarray*}
|u(x)|\leq  {C}{|x|^{-\left(\frac{N-p}{p}+\tau_1\right)}} &  &
\hbox{for}\:|x|> r_1,
\end{eqnarray*}
 for some constants $
\tau_0,\tau_1, r_0,r_1>0$ depending on $N,p,\mu, ||Q||_{\infty}$ and $u$.

  To prove
(\ref{eq:origin growth}) and (\ref{eq:infinity decay}), we regard both
$ u$ and $-u$ as subsolutions of equation \eqref{eq: general equation}
 with function $f$ given by $f=Q|u|^{p^{*}-p}$. Then $f\in
L^{\frac Np}(\R^N)$ and it holds that
\begin{eqnarray*}
|f(x)|\leq {C}{|x|^{-\alpha}} &  & \hbox{for}\:|x|< r_0,\end{eqnarray*} and that
\begin{eqnarray*}|f(x)|\leq{C}{|x|^{-\beta}} &  & \hbox{for}\:|x|> r_1,
\end{eqnarray*}
with $\alpha=\left(\frac{N-p}{p}-\tau_0\right)(p^*-p)<p$ and
 $\beta=\left(\frac{N-p}{p}+\tau_1\right)(p^*-p)>p$.

Therefore, we can apply  Theorem \ref{thm:given growth comparison at origin 1}
 and   Theorem \ref{thm:given growth comparison at infinity 1}  to $\pm u$ in the ball $B_{r_0}$
 and in the exterior ball $B_{r_1}^c$ respectively to conclude that   (\ref{eq:origin growth}) and (\ref{eq:infinity
decay}) hold.
\end{proof}

\begin{proof}[Proof of Theorem
 \ref{thm: inverse of main resuls}] Let $u$  be a nonnegative weak  solution of (\ref{eq:object}) and $Q\in L^{\infty}(\R^N)$ be a nonnegative function.
We  regard $u$ as a nonnegative
supersolution of equation (\ref{eq: general equation})  with function
$f\equiv
 Q|u|^{p^*-p}$.  Then
 \eqref{eq:origin growth 1} and \eqref{eq:infinity decay
2} follow from  Theorem \ref{thm:given growth comparison at origin
2} and Theorem \ref{thm:given growth comparison at infinity 2} respectively.
 This completes the proof of
Theorem \ref{thm: inverse of main resuls}.\end{proof}

\noindent\emph{Acknowledgements.}   The author is financially
supported by the Academy of Finland, project 259224. He wishes to
express his  gratitude to Prof. Xiao Zhong for his guidance in the
preparation of the paper. He also wishes to thank Prof. Daomin Cao
and Shusen Yan for suggesting the problem.  He acknowledges
fruitful discussion with Prof. Daomin Cao  during his one-month
stay in the University of Jyv\"asky\"a in 2011.

\end{document}